\documentclass[11pt]{amsart}
\usepackage{amsmath}
\usepackage{amssymb}
\usepackage{amscd}
\def\NZQ{\mathbb}               % the font for N,Z,Q,R,C
\def\NN{{\NZQ N}}

\def\ZZ{{\NZQ Z}}
\def\RR{{\NZQ R}}
\def\CC{{\NZQ C}}

\def\PP{{\NZQ P}}

\newtheorem{Theorem}{Theorem}[section]
\newtheorem{Lemma}[Theorem]{Lemma}
\newtheorem{Corollary}[Theorem]{Corollary}

\let\epsilon\varepsilon
\let\phi=\varphi
\let\kappa=\varkappa

\begin{document}
\title{Volumes on Complex Analytic Spaces}
\author{Steven Dale Cutkosky }
\thanks{Partially supported by NSF}

\address{Steven Dale Cutkosky, Department of Mathematics,
University of Missouri, Columbia, MO 65211, USA}
\email{cutkoskys@missouri.edu}

%\begin{abstract} 
%\end{abstract}

\maketitle

\section{Introduction} In this paper we show that volumes and related limits exist for line bundles and graded linear series on
compact reduced complex analytic spaces. In Section \ref{SecProp} we define graded linear series associated to a line bundle
and the Kodaira-Iitaka dimension of a graded linear series. Our definition of Kodaira-Iitaka dimension coincides with the classical definitions in \cite{I}, \cite{U} on a compact, normal complex analytic variety.  We restrict to  reduced compact complex analytic spaces since limits of graded linear series do not generally exist on nonreduced projective algebraic varieties, and hence cannot exist in general on compact complex analytic spaces. This  is shown in 
\cite{C2} and \cite{C4}. We now state our main results in this paper.

Suppose that $L$ is a graded linear series on a compact irreducible reduced complex space $X$.
We will call such a space a (compact) complex analytic variety.

The {\it index} $m=m(L)$ of $L$ is defined as the index of groups
$$
m=[\ZZ:G]
$$
where $G$ is the subgroup of $\ZZ$ generated by $\{n\mid L_n\ne 0\}$.

\begin{Theorem}\label{Theorem5*}  Suppose that $X$ is a $d$-dimensional compact complex analytic variety, and $L$ is a graded linear series on $X$ with Kodaira-Iitaka dimension  $\kappa=\kappa(L)\ge 0$. Let $m=m(L)$ be the index of $L$.  Then  
$$
\lim_{n\rightarrow \infty}\frac{\dim_{\CC} L_{nm}}{n^{\kappa}}
$$
exists. 
\end{Theorem}

In particular, from the definition of the index, we have that the limit
$$
\lim_{n\rightarrow \infty}\frac{\dim_{\CC} L_{n}}{{n}^{\kappa}}
$$
exists, whenever $n$ is constrained to lie in an arithmetic sequence $a+bm$ ($m=m(L)$ and $a$ an arbitrary but fixed constant), as $\dim_kL_n=0$ if $m\not\,\mid n$.

An example of a big line bundle where the limit in Theorem \ref{Theorem5*} is an irrational number is given in Example 4 of Section 7 \cite{CS}.

Theorem \ref{Theorem5*} (and the following Corollary \ref{Theorem12*}) have been proven for proper algebraic varieties in a series of papers. In this algebraic setting, it has been proven 
by  Okounkov  \cite{Ok} for section rings of ample line bundles,  Lazarsfeld and Mustata \cite{LM} for section rings of big line bundles, and for graded linear series by Kaveh and Khovanskii \cite{KK}. All of these proofs require the assumption that the ground field is  algebraically closed. The theorem has been proven by the author over a perfect field in \cite{C2}, and over an arbitrary field in \cite{C4}.

In the analytic case, it has been proven by Bouksom \cite{B} when $X$ is a compact K\"ahler manifold and $\mathcal L$ is nef and big.

\begin{Theorem}\label{Theorem100*} Suppose that $X$ is a $d$-dimensional compact complex analytic variety and $L$ is a graded linear series on $X$ with Kodaira-Iitaka dimension $\kappa=\kappa(L)\ge 0$. Let $m=m(L)$ be the index of $L$. 
Let 
$$
Y_{nm}=\mbox{Proj}(\CC[L_{nm}t])\subset \PP_{\CC}^{\dim_{\CC}L_{nm}-1}
$$
be the projective subvariety of $\PP^{\dim_{\CC}L_{nm}-1}$, where $t$ is an indeterminate and $\CC[L_{nm}t]$ is the graded $\CC$-algebra where $t$ has degree 1.
 Let $\deg(Y_{nm})$ be the degree of $Y_{nm}$ in $\PP^{\dim_{\CC}L_{nm}-1}$.
Then  $\dim Y_{nm}=\kappa$ for $n\gg 0$ and 
$$
\lim_{n\rightarrow \infty}\frac{\dim_{\CC} L_{nm}}{n^{\kappa}}=\lim_{n\rightarrow\infty}\frac{\deg(Y_{nm})}{\kappa!n^{\kappa}}.
$$
\end{Theorem}

Here $\deg(Y_{nm})$ is the multiplicity of the  graded $\CC$-algebra $\CC[L_{nm}t]$.

Theorem \ref{Theorem100*} is proven by Kaveh and Khovanskii \cite{KK} when $X$ is a projective variety over  an algebraically closed field (Theorem 3.3 \cite{KK}). We prove the theorem for a proper algebraic variety over an arbitrary field in Theorem 7.2 \cite{C4}.

\begin{Corollary}\label{Theorem12**} Suppose that $X$ is a compact complex analytic variety of dimension $d$ and $\mathcal L$ is a  line bundle on $X$.  Then  the limit
$$
\lim_{n\rightarrow \infty}\frac{\dim_{\CC} \Gamma(X,\mathcal L^n)}{n^d}
$$
exists. 
\end{Corollary}

\begin{Theorem}\label{Theorem18*} Suppose that $X$ is a compact reduced complex analytic space. 
Let $L$ be a graded linear series on $X$ with Kodaira-Iitaka dimension  $\kappa=\kappa(L)\ge 0$.  
 Then there exists a positive integer $r$ such that 
$$
\lim_{n\rightarrow \infty}\frac{\dim_{\CC} L_{a+nr}}{n^{\kappa}}
$$
exists for any fixed $a\in \NN$.
\end{Theorem}
The theorem says that 
$$
\lim_{n\rightarrow \infty}\frac{\dim_{\CC} L_{n}}{n^{\kappa}}
$$
exists if $n$ is constrained to lie in an arithmetic sequence $a+br$ with $r$ as above, and for some fixed $a$. The conclusions of the theorem are a little weaker than the conclusions of Theorem \ref{Theorem5} for  varieties. In particular, the index $m(L)$ has little relevance on reduced but not irreducible schemes (as shown by the example after Theorem 9.2 \cite{C4} and Example 5.5 \cite{C2}).

Theorem \ref{Theorem18*} is proven in Theorem 5.2 \cite{C2} for reduced projective schemes over a perfect field and for reduced proper schemes over an arbitrary field in Theorem 8.2 \cite{C4}.

$m_R$ will denote the maximal ideal of a local ring $R$. $Q(R)$ will denote the quotient field of a domain $R$.
 $\ZZ_+$ denotes the positive integers and $\NN$ the nonnegative integers. 

\section{Cones associated to semigroups}\label{SecCone}
In this section we summarize some results on semigroups and associated cones from \cite{KK}.
Suppose that $S$ is a subsemigroup of $\ZZ^{d}\times \NN$ which is not contained in $\ZZ^d\times\{0\}$. Let $L(S)$ be the subspace of $\RR^{d+1}$ which is generated by $S$, and let $M(S)=L(S)\cap(\RR^d\times\RR_{\ge 0})$. 

Let $\mbox{Con}(S)\subset L(S)$ be the closed convex cone which is the closure of  the set of all linear combinations $\sum \lambda_is_i$ with $s_i\in S$ and $\lambda_i\ge 0$.

$S$ is called {\it strongly nonnegative} (Section 1.4 \cite{KK}) if $\mbox{Cone}(S)$  intersects $\partial M(S)$ only at the origin (this is equivalent to being strongly admissible (Definition 1.9 \cite{KK}) since with our assumptions, $\mbox{Cone}(S)$ is contained in $\RR^d\times\RR_{\ge 0}$,  so the ridge of of $S$ must be contained in $\partial M(S)$). In particular, a subsemigroup of a strongly negative semigroup is itself strongly negative.

We now introduce some notation from \cite{KK}. Let 
\vskip .1truein

$S_k=S\cap (\RR^d\times\{k\})$.

$\Delta(S)=\mbox{Con}(S)\cap (\RR^{d}\times\{1\})$ (the Newton-Okounkov body of $S$).

$q(S)=\dim \partial M(S)$.

$G(S)$ be the subgroup of $\ZZ^{d+1}$ generated by $S$.

$m(S)=[\ZZ:\pi(G(S))]$
  where $\pi:\RR^{d+1}\rightarrow \RR$ be projection onto the last factor.

$\mbox{ind}(S)= [\partial M(S)_{\ZZ}:G(S)\cap \partial M(S)_{\ZZ}]$
where 

$\partial M(S)_{\ZZ}:=\partial M(S)\cap \ZZ^{d+1}= M(S)\cap (\ZZ^d\times\{0\})$.

${\rm vol}_{q(S)}(\Delta(S))$ is the integral volume of $\Delta(S)$. This volume is computed using the translation of the integral measure on $\partial M(S)$.
\vskip .2truein

$S$ is strongly negative if and only if $\Delta(S)$ is a compact set. If $S$ is strongly negative, then the dimension of $\Delta(S)$ is $q(S)$.
\vskip .1truein

\begin{Theorem}\label{ConeTheorem3}(Kaveh and Khovanskii) Suppose that $S$ is strongly nonnegative.  Then 
$$
\lim_{k\rightarrow \infty}\frac{\#S_{m(S)k}}{k^{q(S)}}=\frac{{\rm vol}_{q(S)}(\Delta(S))}{{\rm ind}(S)}.
$$
\end{Theorem}

This is proven in  Corollary 1.16 \cite{KK}. 

With our assumptions, we have that $S_n=\emptyset$ if $m(S)\not\,\mid n$ and  the limit is positive, since
${\rm vol}_{q(S)}(\Delta(S))>0$.

\begin{Theorem}\label{ConeTheorem4}(Kaveh and Khovanskii) Suppose that $q$ is a positive integer such there exists a sequence $k_i\rightarrow \infty$ of positive integers such that  the sequence $\#S_{m(S)k_i}/k_i^q$ is bounded. Then $S$ is strongly nonnegative with $q(S)\le q$.
\end{Theorem}

This is proven in Theorem 1.18 \cite{KK}.

The following theorem generalizes Proposition 3.4 \cite{LM}.

\begin{Theorem}\label{ConeTheorem5}(Theorem 3.3 \cite{C4}) Suppose that $S$ is  strongly nonnegative. Fix $\epsilon>0$. Then there is an integer $p=p_0(\epsilon)$ such that if $p\ge p_0$, then the limit
$$
\lim_{n\rightarrow\infty}\frac{\#(n*S_{pm(S)})}{n^{q(S)}p^{q(S)}}\ge \frac{{\rm vol}_{q(S)}\Delta(S)}{{\rm ind}(S)}-\epsilon
$$
exists, where
$$
n*S_{pm(S)}=\{x_1+\cdots+x_n\mid x_1,\ldots,x_n\in S_{pm(S)}\}.
$$
\end{Theorem}

\section{Compact reduced complex analytic spaces}\label{SecProp}

Suppose that $X$ is a $d$-dimensional compact reduced complex analytic space, and $\mathcal L$ is a line bundle on $X$, by which we will mean
the sheaf of holomorphic sections of a (geometric) complex analytic line bundle.  
The section ring
$$
\bigoplus_{n\ge 0}\Gamma(X,\mathcal L^n)
$$
is a graded $\CC$-algebra. Each $\Gamma(X,\mathcal L^n)$ is a finite dimensional complex vector space since $X$ is compact, and
$\Gamma(X,\mathcal O_X)=\CC$ since $X$ is reduced.
 A graded $\CC$-subalgebra $L=\bigoplus_{n\ge 0}L_n$ of a section ring of a line bundle $\mathcal L$ on $X$ is called a {\it graded linear series} for $\mathcal L$. 

 We define the {\it Kodaira-Iitaka dimension} $\kappa=\kappa(L)$ of a graded linear series $L$  as follows.
Let
$$
\sigma(L)=\max \left\{m\mid 
\begin{array}{l}
 \mbox{there exists $y_1,\ldots,y_m\in L$ which are homogeneous of  positive}\\
\mbox{degree and are algebraically independent over $\CC$}
\end{array}\right\}.
$$
$\kappa(L)$ is then defined as
$$
\kappa(L)=\left\{\begin{array}{ll}
\sigma(L)-1 &\mbox{ if }\sigma(L)>0\\
-\infty&\mbox{ if }\sigma(L)=0
\end{array}\right.
$$

This definition is in agreement with the classical definition for line bundles on normal projective varieties (Definition in Section 10.1 \cite{I} or Chapter 2 \cite{La}), although our $\kappa(L)$ can be smaller on a non normal variety.

In the case that $X$ is also irreducible, let $\CC(X)$ denote the meromorphic function field of $X$.
The algebraic dimension of $X$ is  $a(X):= \mbox{trdeg}_{\CC}\CC(X)$. It is shown in Section 3 of \cite{U} that if $X_1$ and $X_2$ are bimeromorphic then
$a(X_1)=a(X_2)$ and (Theorem 3.1 \cite{U}) that 
\begin{equation}\label{eqC2}
a(X)\le \dim X.
\end{equation}

In the case when $X$ is reduced, with irreducible components $X_1,\ldots,X_s$, we define $a(X)=\max_i\{a_i(X)\}$.

The following lemma generalizes Lemma 2.1 \cite{C2} to reduced analytic spaces.

\begin{Lemma}\label{LemmaKI} Suppose that $L$ is a graded linear series on a $d$-dimensional compact, reduced complex analytic space $X$. Then
\begin{enumerate}\item[1)]
\begin{equation}\label{eqKI1}
\kappa(L)\le a(X)
\end{equation}
\item[2)] There exists a positive constant $\gamma$ such that 
\begin{equation}\label{eqKI4}
\dim_{\CC} L_n<\gamma n^{a(X)}
\end{equation}
for all $n$.
\item[3)] Suppose that $\kappa(L)\ge 0$. Then there exists a positive constant $\alpha$ and a positive integer $e$ such that 
\begin{equation}\label{eqKI2}
\dim_{\CC}L_{en}>\alpha n^{\kappa(L)}
\end{equation}
for all positive integers $n$. 
\item[4)] Suppose that  $L$ is a graded linear series on $X$. Then $\kappa(L)=-\infty$ if and only if $L_n=0$ for all $n>0$.
\end{enumerate}
\end{Lemma}

\begin{proof} We first prove 1). Suppose that $L_n\ne 0$ for some $n>0$. There there exists a positive integer $e$ such that if
$$
B=\bigoplus_{n\ge 0}B_n:=\CC[L_et^e]\subset \sum L_nt^n=L
$$
then $\sigma(B)=\sigma(L)$. We have that 
$$
\kappa(B)=\mbox{Krull dimension}(B)-1
$$
by Lemma 8.1 \cite{C2}. 

First assume that $X$ is irreducible. Then 
$$
\mbox{Krull dimension}(B)=\mbox{trdeg}_{\CC}Q(B).
$$
Suppose that $f_0,\ldots,f_r$ generate $L_e$ as a complex vector space. Then 
$$
Q(B)=K(f_0t^e)
$$
where 
$$
K=\CC(\frac{f_1}{f_0},\ldots,\frac{f_r}{f_0})\subset \CC(X).
$$
Hence 
$$
\mbox{trdeg}_{\CC}Q(B)\le a(X)+1,
$$
and 1) follows. 

Now assume that $X$ is reduced but not irreducible. Let $X_i$ for $1\le i\le s$ be the irreducible components of $X$. We have (since $X$ is reduced) inclusions of $\mathcal O_X$-modules
$$
0\rightarrow \mathcal O_X\rightarrow \bigoplus_{i=1}^s\mathcal O_{X_i},
$$
giving us inclusions
$$
B_n\subset L_{ne}\subset \Gamma(X,\mathcal L^{ne})\subset \bigoplus_{i=1}^s\Gamma(X_i,(\mathcal L|X_i)^{ne})
$$
for all $n$. Let $B_n^i$ be the image of $B_n$ in $\Gamma(X_i,(\mathcal L|X_i)^{ne})$. Then $B^i:=\bigoplus_{n\ge 0}B_n^i$ is a graded subalgebra of $\bigoplus_{n\ge 0}\Gamma(X_i,(\mathcal L|X_i)^{ne})$, so it is a domain.
Let $P_i$ be the prime ideal which is the kernel of the surjection $B\rightarrow B^i$. The natural graded homomorphism $B\rightarrow \bigoplus_{i=1}^sB^i$ is 1-1, so $\cap P_i=0$. Since every prime ideal of $B$ must contain one of the $P_i$, we have that
$$
\mbox{Krull Dimension}(B)=\max_i\{\mbox{Krull Dimension}(B^i)\}.
$$
Thus, since the $X_i$ are reduced irreducible,
$$
\mbox{Krull Dimension}(B)\le\max_i\{a(X_i)+1\}
$$
and $\kappa(L)\le a(X)$.

We now prove 2). If $X$ is normal and irreducible, the desired bound is proven in Theorem 8.1 \cite{U}. Let $X_i'$ be the normalizations of the irreducible components $X_i$ of $X$. Let $X'$ be the disjoint union of the $X_i'$, and $\pi:X'\rightarrow X$ be the natural normalization map. Tensoring the inclusion
$$
0\rightarrow \mathcal O_X\rightarrow \pi_*\mathcal O_{X'}
$$
with $\mathcal L^n$, we obtain inclusions
$$
0\rightarrow \Gamma(X,\mathcal L^n)\rightarrow \Gamma(X',\pi^*\mathcal L^n)\cong \bigoplus_{i=1}^s\Gamma(X_i',(\pi^*\mathcal L^n|X_i')).
$$
The desired bound now follows from the equalities $a(X_i')=a(X_i)$, and Lemma 5.5 and the fact that the desired bound holds on normal varieties.

The proofs of 3) and 4) follows from the proofs of 3) and 4) in \cite{C2}.
\end{proof}

\section{Limits of graded linear series on compact  complex analytic varieties}\label{SecLim}

Suppose that $L$ is a graded linear series on a compact irreducible reduced complex space $X$.
We will call such a space a (compact) complex analytic variety.

The {\it index} $m=m(L)$ of $L$ is defined as the index of groups
$$
m=[\ZZ:G]
$$
where $G$ is the subgroup of $\ZZ$ generated by $\{n\mid L_n\ne 0\}$.

The following theorem (and the following Corollary \ref{Theorem12*}) have been proven for proper algebraic varieties in a series of papers. In this algebraic setting, it has been proven 
by  Okounkov  \cite{Ok} for section rings of ample line bundles,  Lazarsfeld and Mustata \cite{LM} for section rings of big line bundles, and for graded linear series by Kaveh and Khovanskii \cite{KK}. All of these proofs require the assumption that the ground field  is  algebraically closed. The theorem has been proven by the author over a perfect field in \cite{C2}, and over an arbitrary field in \cite{C4}.

In the analytic case, it has been proven by Bouksom \cite{B} when $X$ is a compact K\"ahler manifold and $\mathcal L$ is nef and big.

\begin{Theorem}\label{Theorem5}  Suppose that $X$ is a $d$-dimensional compact complex analytic variety, and $L$ is a graded linear series on $X$ with Kodaira-Iitaka dimension  $\kappa=\kappa(L)\ge 0$. Let $m=m(L)$ be the index of $L$.  Then  
$$
\lim_{n\rightarrow \infty}\frac{\dim_{\CC} L_{nm}}{n^{\kappa}}
$$
exists. 
\end{Theorem}

In particular, from the definition of the index, we have that the limit
$$
\lim_{n\rightarrow \infty}\frac{\dim_{\CC} L_{n}}{{n}^{\kappa}}
$$
exists, whenever $n$ is constrained to lie in an arithmetic sequence $a+bm$ ($m=m(L)$ and $a$ an arbitrary but fixed constant), as $\dim_kL_n=0$ if $m\not\,\mid n$.

An example of a big line bundle where the limit in Theorem \ref{Theorem5} is an irrational number is given in Example 4 of Section 7 \cite{CS}.

The following theorem is proven by Kaveh and Khovanskii \cite{KK} when $X$ is a projective variety over  an algebraically closed field (Theorem 3.3 \cite{KK}). We prove the theorem for a proper algebraic variety over an arbitrary field in Theorem 7.2 \cite{C4}.

\begin{Theorem}\label{Theorem100} Suppose that $X$ is a $d$-dimensional compact complex analytic variety and $L$ is a graded linear series on $X$ with Kodaira-Iitaka dimension $\kappa=\kappa(L)\ge 0$. Let $m=m(L)$ be the index of $L$. 
Let 
$$
Y_{nm}=\mbox{Proj}(\CC[L_{nm}t])\subset \PP_{\CC}^{\dim_{\CC}L_{nm}-1}
$$
be the projective subvariety of $\PP^{\dim_{\CC}L_{nm}-1}$, where $t$ is an indeterminate and $\CC[L_{nm}t]$ is the graded $\CC$-algebra where $t$ has degree 1.
 Let $\deg(Y_{nm})$ be the degree of $Y_{nm}$ in $\PP^{\dim_{\CC}L_{nm}-1}$.
Then  $\dim Y_{nm}=\kappa$ for $n\gg 0$ and 
$$
\lim_{n\rightarrow \infty}\frac{\dim_{\CC} L_{nm}}{n^{\kappa}}=\lim_{n\rightarrow\infty}\frac{\deg(Y_{nm})}{\kappa!n^{\kappa}}.
$$
\end{Theorem}

Here $\deg(Y_{nm})$ is the multiplicity of the  graded $\CC$-algebra $\CC[L_{nm}t]$. 

We now proceed to prove Theorems \ref{Theorem5} and \ref{Theorem100}.

Let $Q$ be a nonsingular point of  $X$. 
Let  $R=\mathcal O_{X,Q}$. $R$ is a $d$-dimensional regular local ring. 

Choose a regular system of parameters $y_1,\ldots, y_d$ in $R$. By a similar  argument to that of the proof of Theorem 4.1 \cite{C2}
(or Theorem 7.1 \cite{C4}), we may
define a valuation $\nu$ of the quotient field $Q(R)$ of $R$ dominating $R$, by stipulating that
 \begin{equation}\label{eq20}
\nu(y_i)= e_i\mbox{ for $1\le i\le d$}
\end{equation}
where $\{e_i\}$ is the standard basis of the totally ordered  group $\Gamma_{\nu}=(\ZZ^d)_{\rm lex}$, and
 $\nu(c)=0$ if $c$ is a unit in $R$. As  in the proof of Theorem 4.1 \cite{C2}, we have that the residue field of the valuation ring $V_{\nu}$ of $\nu$ is  $V_{\nu}/m_{\nu}=\CC$ (this follows most directly since $\nu$ is zero dimensional and $\CC$ is algebraically closed).

$L$ is a graded linear series for some line bundle $\mathcal L$ on $X$.  
There exists an $R=\mathcal O_{X,Q}$-module isomorphism $\sigma:\mathcal L_Q\rightarrow R$. We thus have an isomorphism of graded $R$-algebras
$$
\bigoplus_{n\ge 0} \mathcal L^n_Q\stackrel{\cong}{\rightarrow} R[t]
$$
to a (standard graded) polynomial ring over $R$. The restriction maps $\Gamma(X,\mathcal L^n)\rightarrow \mathcal L^n_Q$
are 1-1 since $X$ is irreducible and reduced, so we have a 1-1 graded $\CC$-algebra homomorphism
$$
\bigoplus_{n\ge 0}\Gamma(X,\mathcal L^n)\rightarrow R[t].
$$

We have an induced $\CC$-algebra homomorphism 
$$
 L\rightarrow R\subset V_{\nu}
 $$
 defined by mapping $t$ to 1.

Given a nonnegative element $\gamma$ in the  value group $\Gamma_{\nu}=(\ZZ^d)_{\rm lex}$ of $\nu$, we have associated valuation ideals $I_{\gamma}$ and $I_{\gamma}^+$ in $V_{\nu}$ defined by 
$$
I_{\gamma}=\{f\in V_{\nu}\mid \nu(f)\ge \gamma\}
$$
and
$$
I_{\gamma}^+=\{f\in V_{\nu}\mid \nu(f)>\gamma\}.
$$
Since $V_{\nu}/m_{\nu}=\CC$,  we have   that 
\begin{equation}\label{eqC1}
I_{\lambda}/I_{\lambda}^+\cong \CC
\end{equation}
for all nonnegative elements $\lambda\in \Gamma_{\nu}$. 
  Let
\begin{equation}\label{eqR1}
S(L)_n=\{\gamma\in \Gamma_{\nu}\mid \mbox{ there exists $f\in L_n$ such that $\nu(f)=\gamma$}\}.
\end{equation}
By (\ref{eqC1}), we have that
\begin{equation}\label{eqR3}
\dim_{\CC} L_n\cap I_{\gamma}/L_n\cap I_{\gamma}^+=\left\{
\begin{array}{ll} 1&\mbox{ if there exists $f\in L_n$ with $\nu(f)=\gamma$}\\
0&\mbox{ otherwise.}
\end{array}
\right.
\end{equation}

Since every element of $L_n$ has non negative value (as $L_n\subset V_{\nu}$), we have by (\ref{eqR1}) and  (\ref{eqR3}) that 
\begin{equation}\label{eqR2}
\dim_kL_n=\#S(L)_n
\end{equation}
for all $n$.
Let 
$$
S(L)=\{(\gamma,n)|\gamma \in S(L)_n\}.
$$
$S(L)$ is a subsemigroup of $\ZZ^{d+1}$.

We have that $m=m(L)=m(S(L))$. Let $q(L)=q(S(L))$.

By (\ref{eqKI4}) of Lemma \ref{LemmaKI} and (\ref{eqC2}), there exists a positive constant $\gamma$ such that 
$$
\dim_{\CC}L_n\le \gamma n^d
$$
for all $n$. By (\ref{eqR2}) and \ref{ConeTheorem4}, we have that $S(L)$ is a strongly nonnegative semigroup.

by Theorem \ref{ConeTheorem3} and (\ref{eqR2}), we have that
\begin{equation}\label{eqnr80}
\lim_{n\rightarrow \infty}\frac{\dim_{\CC}L_{nm}}{n^{q(L)}}=\lim_{n\rightarrow \infty}\frac{\#(S(L)_{nm})}{n^{q(L)}}
=\frac{{\rm vol}_{q(L)}(\Delta(S(L)))}{{\rm ind}(S(L))}
\end{equation}
exists.

Let $Y_{pm}$ be the varieties defined in the statement of Theorem \ref{Theorem100}. Let $d(pm)=\dim Y_{pm}$. The coordinate ring of $Y_{pm}$ is the $\CC$-subalgebra $L^{[pm]}:= \CC[L_{pm}]$ of $L$ (but with the grading giving elements of $L_{pm}$ degree 1). The Hilbert polynomial $P_{Y_{pm}}(n)$ of $Y_{pm}$ (Section I.7 \cite{H} or Theorem 4.1.3 \cite{BH}) has the properties that
\begin{equation}\label{eqred90}
P_{Y_{pm}}(n)=\frac{\deg(Y_{pm})}{d(pm)!}n^{d(pm)}+\mbox{lower order terms}
\end{equation}
and
\begin{equation}\label{eqred61}
\dim_{\CC}L^{[pm]}_{npm}=P_{Y_{pm}}(n)
\end{equation}
for $n\gg 0$. We have that
\begin{equation}\label{eqnr81}
\lim_{n\rightarrow \infty}\frac{\dim_{\CC}(L^{[pm]})_{npm}}{n^{d(pm)}}=\frac{\deg(Y_{pm})}{d(pm)!}.
\end{equation}

 For $p$ sufficiently large, we have that $m(S(L^{[pm]}))=mp$. Let $\mathcal C$ be the closed cone generated by $S(L)_{pm}$ in $\RR^{d+1}$. We also have that 
$$
\dim (\mathcal C\cap (\RR^d\times\{1\}))=\dim(\Delta(S(L))=q(L)
$$
 for $p$ sufficiently large. Since
$S(L)_{pm}=S(L^{[pm]})_{pm}$, we have that 
$$
\dim (\mathcal C\cap (\RR^d\times\{1\}))\le \dim(\Delta(S(L^{[pm]}))\le \dim (\Delta(L)).
$$
Thus 
\begin{equation}\label{eqred30}
q(S(L^{[pm]}))=q(L)
\end{equation}
for all $p$ sufficiently large.

By the definition of Kodaira-Iitaka dimension, we also have that
\begin{equation}\label{eqred31} 
\kappa(L^{[pm]})=\kappa(L)
\end{equation}
for $p$ sufficiently large.

Now by graded Noether normalization (Section I.7 \cite{H} or Theorem 1.5.17\cite{BH}), the finitely generated  $\CC$-algebra $L^{[pm]}$ satisfies
\begin{equation}\label{eqred63}
d(pm)=\dim Y_{pm}=\mbox{Krull dimension}(L^{[pm]})-1=\kappa(L^{[pm]}).
\end{equation}

 We have that
\begin{equation}\label{eqred62}
\dim_{\CC} L^{[pm]}_{npm}= \#(S(L^{[pm]})_{npm})
\end{equation}
for all  $n$. $S(L^{[pm]})$ is strongly nonnegative since $S(L^{[pm]})\subset S(L)$ (or since $L^{[pm]}$ is a finitely generated $\CC$-algebra).
It follows from Theorem \ref{ConeTheorem3}, (\ref{eqred62}), (\ref{eqred61}), (\ref{eqred90}) and (\ref{eqred63})  that 
\begin{equation}\label{eqred91}
q(S(L^{[pm]}))=d(pm)=\kappa(L^{[pm]}).
\end{equation}
From (\ref{eqred30}), (\ref{eqred91}) and (\ref{eqred31}), we have that 
\begin{equation}\label{eqnr85}
q(L)=\kappa(L)=\kappa.
\end{equation}

Theorem \ref{Theorem5} now follows from (\ref{eqnr80}) and (\ref{eqnr85}).
We now prove Theorem \ref{Theorem100}.  For all $p$, we have inequalities
$$
\#(n*S(L)_{mp})\le \#(S(L^{[mp]})_{nmp})\le \#(S(L)_{nmp}).
$$
The second term in the inequality is $\dim_{\CC}(L^{[pm]})_{nmp}$ and the third term is $\dim_{\CC}L_{nmp}$.
Dividing by $n^{\kappa}p^{\kappa}$, and taking the limit as $n\rightarrow \infty$, we obtain from 
Theorem \ref{ConeTheorem4}, (\ref{eqnr85}) and (\ref{eqnr80}) for the first term and (\ref{eqnr81}), (\ref{eqred31}) and (\ref{eqred63}) for the second term,
 that
for given $\epsilon >0$, we can take $p$ sufficiently large that
$$
\lim_{n\rightarrow\infty}\frac{\dim_{\CC}L_{nm}}{n^{\kappa}}-\epsilon\le \frac{\deg(Y_{pm})}{\kappa!p^{\kappa}}\le 
\lim_{n\rightarrow\infty}\frac{\dim_{\CC}L_{nm}}{n^{\kappa}}.
$$
Taking the limit as $p$ goes to infinity then proves Theorem \ref{Theorem100}.

\begin{Corollary}\label{Theorem12*} Suppose that $X$ is a compact complex analytic variety of dimension $d$ and $\mathcal L$ is a  line bundle on $X$.  Then  the limit
$$
\lim_{n\rightarrow \infty}\frac{\dim_{\CC} \Gamma(X,\mathcal L^n)}{n^d}
$$
exists. 
\end{Corollary}

\begin{proof} Let $L=\bigoplus_{n\ge 0}\Gamma(X,\mathcal L^n)$. In the case when $\kappa(L)<d$ the corollary is immediate from
Theorem \ref{Theorem5} (the limit is zero). Suppose that $\kappa(L)=d$. We must show that $m(L)=1$, and then the limit follows from Theorem \ref{Theorem5}. 

Since $\kappa(L)=d$, there exists $\tau>0$ and $e>0$ such that $\dim_{\CC}\Gamma(X,\mathcal L^{ne})>\tau n^d$ for all $n$ sufficiently large.
By a theorem of Moishezon, there exist a proper modification $\pi:X'\rightarrow X$ such that $X'$ is a nonsingular projective variety
(Chapter II \cite{Mo}, Theorem 3.6 \cite{U}).
Let $\mathcal L'=\pi^*(\mathcal L)$. 

We have exact sequences
$$
0\rightarrow \mathcal O_X\rightarrow \pi_*\mathcal O_{X'}\rightarrow \mathcal F\rightarrow 0
$$
where $\mathcal F$ is a coherent $\mathcal O_X$ module whose support has dimension less than $d$. Tensoring with $\mathcal L^{ne}$ and taking global sections, we have exact sequences
\begin{equation}\label{eqC6}
0\rightarrow \Gamma(X,\mathcal L^{n})\rightarrow \Gamma(X,\pi^*(\mathcal L)^{n})\rightarrow \Gamma(X,\mathcal F\otimes \mathcal L^{n}).
\end{equation}
There exists a constant $\lambda$ such that 
\begin{equation}\label{eqC5}
\dim_{\CC}\Gamma(X,\mathcal F\otimes \mathcal L^{n})<\lambda n^{d-1}
\end{equation}
for all $n$, since the support of $\mathcal F$ has dimension $\le d-1$.
Thus we have a positive constant $\tau'$ such that
$\dim_{\CC}\Gamma(X',\pi^*(\mathcal L)^{ne})>\tau' n^d$ for all $n$ sufficiently large. Thus $\pi^*(\mathcal L)$ is a big line bundle on $X'$. Let $L'=\bigoplus_{n\ge 0}\Gamma(X,\pi^*\mathcal L^n)$. It follows from Lemma 2.2 \cite{LM} that $m(L')=1$, so that $m(L)=1$
by (\ref{eqC6}) and (\ref{eqC5}).
 \end{proof}

\section{Limits on compact reduced complex analytic spaces}\label{SecRed}

Suppose that $X$ is a compact complex analytic space and $L$ is a graded linear series for a line bundle $\mathcal L$ on $X$.
Suppose that $Y$ is a closed analytic subspace of $X$. Set $\mathcal L|Y=\mathcal L\otimes_{\mathcal O_X}\mathcal O_Y$. Taking global sections of the natural surjections
$$
\mathcal L^n\stackrel{\phi_n}{\rightarrow} (\mathcal L|Y)^n\rightarrow 0,
$$
for $n\ge 1$ we have  induced short exact sequences of $\CC$-vector spaces
\begin{equation}\label{eq54}
0\rightarrow K(L,Y)_n\rightarrow L_n\rightarrow (L|Y)_n\rightarrow 0,
\end{equation}
where 
$$
(L|Y)_n:=\phi_n(L_n)\subset \Gamma(Y,({\mathcal L}|Y)^n)
$$
 and $K(L ,Y)_n$ is the kernel of $\phi_n|L_n$. Defining $K(L,Y)_0=\CC$ and $(L|Y)_0=\phi_0(L_0)$, we have that
$L|Y=\bigoplus_{n\ge 0}(L|Y)_n$ is a graded linear series for $\mathcal L|Y$ and $K(L,Y)=\bigoplus_{n\ge 0}K(L,Y)_n$ is a graded linear series for $\mathcal L$.

The following lemma is proven in Lemma 5.1 \cite{C2} for proper algebraic varieties. The proof is the same for analytic spaces.

\begin{Lemma}\label{Lemma50a} Suppose that $X$ is a compact reduced complex analytic space  and $X_1,\ldots,X_s$ are the irreducible components of $X$. Suppose that $L$ is a graded linear series on $X$. Then 
$$
\kappa(L)=\max\{\kappa(L|X_i)\mid 1\le i\le s\}.
$$
\end{Lemma}

The following theorem is proven in Theorem 5.2 \cite{C2} for reduced projective schemes over a perfect field and for reduced proper schemes over an arbitrary field in Theorem 8.2 \cite{C4}.

\begin{Theorem}\label{Theorem18} Suppose that $X$ is a compact reduced complex analytic space. 
Let $L$ be a graded linear series on $X$ with Kodaira-Iitaka dimension  $\kappa=\kappa(L)\ge 0$.  
 Then there exists a positive integer $r$ such that 
$$
\lim_{n\rightarrow \infty}\frac{\dim_{\CC} L_{a+nr}}{n^{\kappa}}
$$
exists for any fixed $a\in \NN$.
\end{Theorem}
The theorem says that 
$$
\lim_{n\rightarrow \infty}\frac{\dim_{\CC} L_{n}}{n^{\kappa}}
$$
exists if $n$ is constrained to lie in an arithmetic sequence $a+br$ with $r$ as above, and for some fixed $a$. The conclusions of the theorem are a little weaker than the conclusions of Theorem \ref{Theorem5} for  varieties. In particular, the index $m(L)$ has little relevance on reduced but not irreducible schemes (as shown by the example after Theorem 9.2 \cite{C4} and Example 5.5 \cite{C2}).

\begin{proof}  Let $X_1,\ldots,X_s$ be the irreducible components of $X$. 
 Define graded linear series $M^i$ on $X$
 by $M^0=L$, $M^i=K(M^{i-1},X_i)$ for $1\le i\le s$. 
 By (\ref{eq54}), for $n\ge 1$, we have exact sequences of $k$-vector spaces 
 $$
 0\rightarrow (M^{j+1})_n=K(M^j,X_{j+1})_n\rightarrow M_n^j\rightarrow (M^j|X_{j+1})_n\rightarrow 0
 $$
 for $0\le j\le s-1$, and thus
 $$
 M_n^j={\rm Kernel}(L_n\rightarrow \bigoplus_{i=1}^j(L|X_i)_n)
 $$
 for $1\le j\le s$. The natural map $L\rightarrow \bigoplus_{i=1}^sL|X_i$ is an injection of $\CC$-algebras since 
 $X$ is reduced. Thus $M_n^s=(0)$, and
 \begin{equation}\label{eq71}
 \dim_{\CC}L_n=\sum_{i=1}^{s}\dim_{\CC} (M^{i-1}|X_i)_n
 \end{equation}
 for all $n$.
  Let $r=\mbox{LCM}\{m(M^{i-1}|X_i)\mid  \kappa(M^{i-1}|X_i)=\kappa(L)\}$. 
 The theorem now follows from Theorem \ref{Theorem5} applied to each of the $X_i$ with $\kappa(M^{i-1}|X_i)=\kappa(L)$ (we can start with an $X_1$ with $\kappa(L|X_1)=\kappa(L)$).
\end{proof}

\end{document}